\documentclass[12pt]{amsart}
\pdfoutput=1
\usepackage[left=2.4cm,right=2.4cm,top=3cm,bottom=3.5cm]{geometry}
\usepackage{graphicx}
\usepackage{url}

\usepackage{caption}
\usepackage[bookmarks=true,%
    colorlinks=true,%
    linkcolor=blue,%
    citecolor=blue,%
    filecolor=blue,%
    menucolor=blue,%
    urlcolor=blue,%
    breaklinks=true]{hyperref}

\newtheorem*{theorem}{Theorem}
\newtheorem*{lemma}{Lemma}

\newcommand\BH{\mathbb{H}}

\title[Geodesic triangulations exist for cusped Platonic manifolds]{Geodesic triangulations exist\\ for cusped Platonic manifolds}

\author{Matthias Goerner}
\email{enischte@gmail.com}
\address{Pixar Animation Studios, 1200 Park Avenue, Emeryville, CA94608, USA}
\urladdr{http://unhyperbolic.org/}

\begin{document}

\begin{abstract}
We show that if a cusped hyperbolic manifold is Platonic, i.e., can be decomposed into isometric Platonic solids, it can also be decomposed into geodesic ideal tetrahedra.
\end{abstract}

\maketitle

It is an open question whether every cusped hyperbolic 3-manifold posseses a triangulation by ideal geodesic  non-flat tetrahedra. This has been proven under relaxed conditions such as allowing some tetrahedra to be flat \cite{PetronioWeeks:partiallyFlatTrig} or taking a cover \cite{LuoSchleimerTillmann:virtualGeometricTrig}. Here, we show it for cusped hyperbolic Platonic manifolds which were introduced in \cite{goerner:platCensus} to be manifolds that can be decomposed into isometric Platonic solids:

\begin{theorem}
An ideal geodesic cell decomposition of a cusped hyperbolic 3-manifold can be subdivided into ideal non-flat tetrahedra if any of the following hold:
\begin{enumerate}
\item each face of each cell is a triangle, or
\item all cells are cubes, or
\item all cells are dodecahedra,
\end{enumerate}
where the cubes/dodecahedra need not to be regular. In particular, every cusped hyperbolic Platonic manifold has an ideal geometric triangulation.
\end{theorem}

\begin{lemma}
Let $P$ be a convex ideal polyhedron in $\BH^3$. Consider a choice of non-intersecting face diagonals subdividing each face into triangles. The polyhedron $P$ can be subdivided into non-flat ideal tetrahedra yielding the chosen face diagonals if there is a vertex $v$ of $P$ such that the diagonals on the faces touching $v$ all meet at $v$.
\end{lemma}

\begin{proof}
Consider the 2-cell complex obtained by subdividing the polyhedron's surface along the diagonals and remove all cells that are adjacent to $v$. Coning this 2-cell complex to $v$ yields a subdivision of the polyhedron.
\end{proof}
The lemma and its proof were inspired by \cite{LuoSchleimerTillmann:virtualGeometricTrig}. While it is clear that each polyhedron of a decomposition into geodesic cells (such as the canonical cell decomposition \cite{EP}) can be subdivided into tetrahedra individually, the challenge is to do so such that the face diagonals introduced in the process are compatible with the face gluings.

\newcounter{diagonalAssignmentSteps}

\begin{proof}[Proof of theorem]
The case where each face of each cell is a triangle immediately follows from the lemma since no diagonals need to be chosen to cut a face into a triangle. The case where each cell is a cube was proven in \cite[Appendix]{goerner:platCensus}.\\
Let such a 3-manifold be a union $P_1\cup\dots\cup P_n$ of dodecahedra. Similarly to the case of the cube, we call a sequence $F_0,F_1,F_2,\dots,F_{k-1}$ of distinct faces in the decomposition $P_1\cup\dots\cup P_n$ a face cycle if $F_i$ and $F_{i+1}$ are opposite faces of the same dodecahedron for each $i=0,\dots,k-1$ (where indexing is cyclic so $F_0=F_k).$ Note that the reverse $F_{k-1},F_{k-2},\dots,F_0$ is also a face cycle, so the faces of the dodecahedral decomposition naturally partition into unoriented face cycles. Let us pick an orientation for each face cycle (which can also be thought of as picking an orientation for each of the circles that the dual 1-skeleton naturally splits into given that there is a notion of two dual edges touching a dual vertex in opposite directions induced from the notion of faces of a dodecahedron being opposite).\\
Such a choice of orientation induces an assignment of each 2-cell $F$ of the decomposition to one of the two representatives of $F$, i.e., to a face $f$ of one of the dodecahedra, say $P$. If a 2-cell was assigned to face $f$ of $P$, we say that $P$ owns $f$. Each dodecahedron owns exactly one face of any pair of opposite faces.\\
The idea is that if $P$ owns $f$, $P$ will control the choice of diagonals on $f$ and thus on the corresponding face $F$ of the decomposition.
We only allow arrangements of diagonals on a face so that all diagonals meet in one vertex of the face. We choose diagonals fulfilling the condition of the lemma in 5 steps:
\begin{enumerate}
\item For each dodecahedron owning three faces with a common vertex, we pick such three faces and choose the diagonals on these three faces to meet in the common vertex. For all other faces such a dodecahedron owns, choose diagonals randomly. \label{enum:stepOne}
\setcounter{diagonalAssignmentSteps}{\value{enumi}}
\end{enumerate}
\begin{figure}[h]
\captionsetup{width=7cm}
\begin{center}
\begin{minipage}[b]{7cm}
\begin{center}
\includegraphics[width=5.5cm]{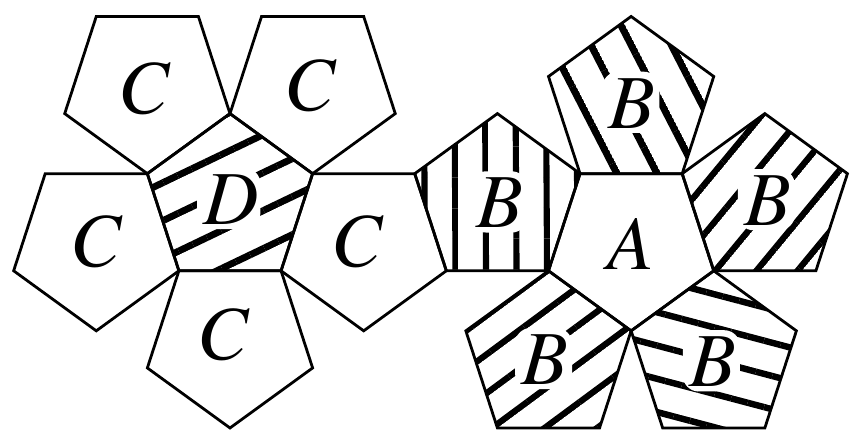}
\end{center}
\end{minipage} \hspace{1cm}
\begin{minipage}[b]{7cm}
\begin{center}
\includegraphics[width=4cm]{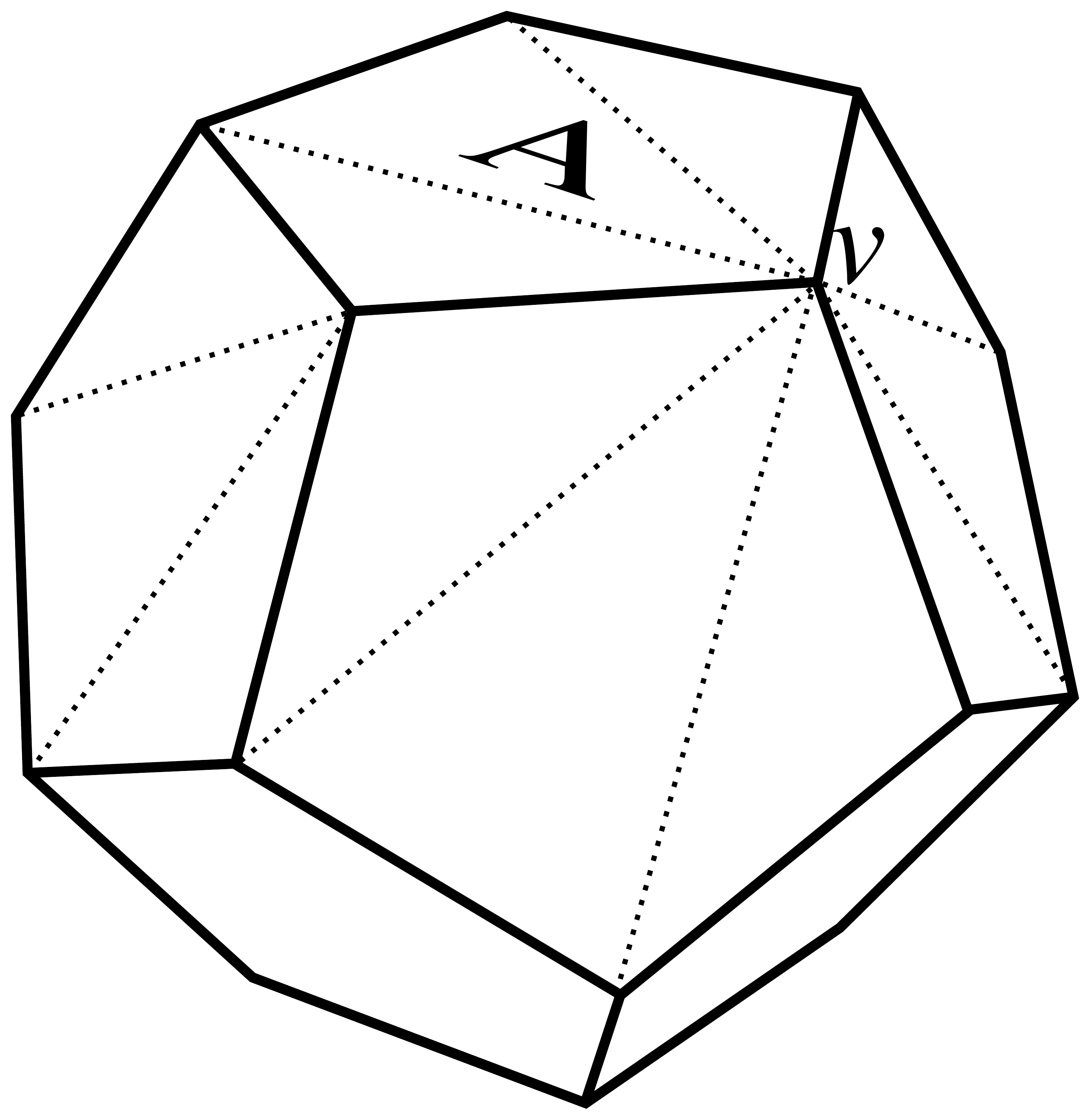}
\end{center}
\end{minipage}\\
\begin{minipage}[t]{7cm}
\begin{center}
\caption{Each dodecahedron not covered in Step~\ref{enum:stepOne} owns the shaded faces. We label the faces of such a dodecahedron as indicated.\label{fig:ownedPattern}}
\end{center}
\end{minipage} \hspace{1cm}
\begin{minipage}[t]{7cm}
\begin{center}
\caption{Choosing diagonals on $B$-faces when diagonals on the $A$-face have already been chosen to meet at vertex $v$. The choice for the three $B$-faces not touching $v$ does not matter.\label{fig:AalreadyChosen}}
\end{center}
\end{minipage}
\end{center}
\end{figure}
From now on, we only consider dodecahedra not covered in Step~\ref{enum:stepOne}.\\
Claim: Up to combinatorial symmetry, the pattern of faces owned by such a dodecahedron $P$ is always as indicated in Figure~\ref{fig:ownedPattern}. To see this, pick one face $f$ owned by $P$. We have two possible cases: $P$ owns none of the faces neighboring $f$ resulting in the pattern indicated in Figure~\ref{fig:ownedPattern} with $f$ corresponding to the face labeled $D$. $P$ owns at least one of the faces neighboring $f$, in which case the constraints (namely, $P$ cannot own three faces with a common vertex and owns exactly one face of any pair of opposite faces) quickly yield again the pattern indicated in Figure~\ref{fig:ownedPattern} with $f$ corresponding to a face labeled $B$.
\begin{enumerate}
\setcounter{enumi}{\value{diagonalAssignmentSteps}}
\item We choose diagonals randomly on all $D$-faces. \label{enum:stepDfaces}
\setcounter{diagonalAssignmentSteps}{\value{enumi}}
\end{enumerate}
Now we are left with choosing diagonals for all $B$-faces. Once we have made a decision about the diagonals on a face glued to the $A$-face of a dodecahedron $P$, we can safely choose diagonals on all of the $B$-faces of $P$ as in Figure~\ref{fig:AalreadyChosen} so that $P$ will fulfill the condition of the lemma. We might have chosen diagonals on some faces glued to $A$-faces already in Step~\ref{enum:stepOne} or \ref{enum:stepDfaces}. We can pick diagonals for the $B$-faces adjacent to those $A$-faces safely and then recurse, i.e., whenever we have picked diagonals on a face that is glued to an $A$-face of a dodecahedron $P$, we choose diagonals for the $B$-faces of $P$.\\
However, there might a $B$-face, say $f_0$, that this procedure will not reach since $f_0$ is adjacent to an $A$-face that is glued to a $B$-face that is adjacent to an $A$-face glued to a $B$-face, and so on until we are back at the $B$-face $f_0$. We will now describe how to resolve such dependency cycles before starting the above recursion.
\begin{figure}[h]
\captionsetup{width=7cm}
\begin{center}
\begin{minipage}[b]{7cm}
\begin{center}
\includegraphics[width=3.5cm]{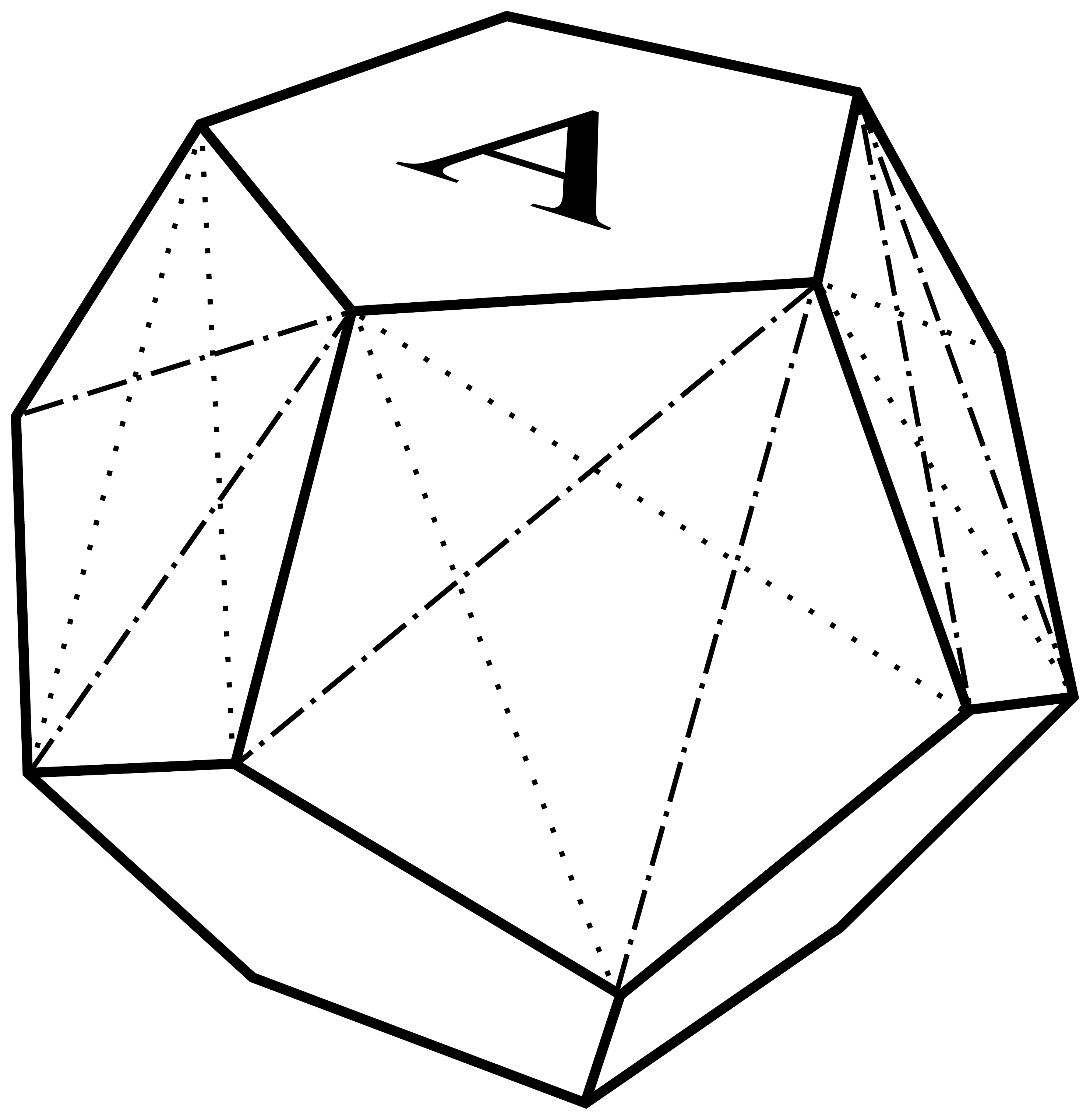}
\end{center}
\end{minipage}
\hspace{1.5cm}
\begin{minipage}[b]{7cm}
\begin{center}
\includegraphics[width=5.25cm]{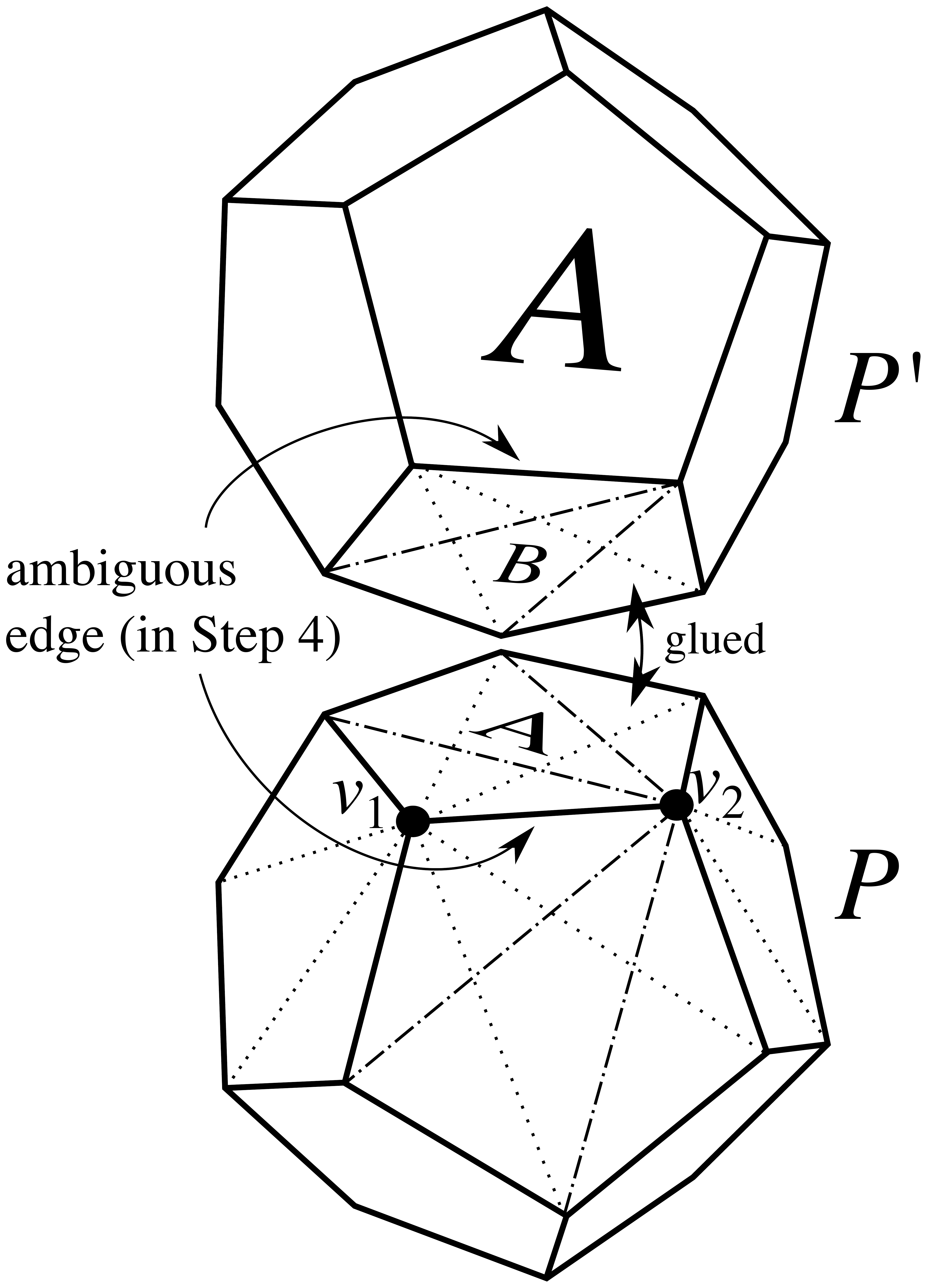}
\end{center}
\end{minipage}\\
\begin{minipage}[t]{7cm}
\begin{center}
\caption{We restrict the possible arrangements of diagonals on a $B$-face to the ones shown.\label{fig:allowedBdiagonals}}
\end{center}
\end{minipage}
\hspace{1.5cm}
\begin{minipage}[t]{7cm}
\caption{We choose diagonals on four $B$-faces of a dodecahedron if the $A$-face is glued to a $B$-face.\label{fig:twoDode}}
\end{minipage}
\end{center}
\end{figure}\\
We restrict the possible choices on each $B$-face to the two arrangements of diagonals for which the diagonals meet in a vertex of the $A$-face, see Figure~\ref{fig:allowedBdiagonals}, and perform the following step:
\begin{enumerate}
\setcounter{enumi}{\value{diagonalAssignmentSteps}}
\item For each $A$-face of a dodecahedron $P$ glued to a $B$-face of a (not necessarily distinct) dodecahedron $P'$, we do the following: the above restriction on the $B$-faces gives two possible arrangements of diagonals on the $A$-face of $P$ under the face gluing, see Figure~\ref{fig:twoDode}. Let $v_1$ and $v_2$ be the two vertices of $P$ where the diagonals of one or the other arrangement meet. For each $B$-face of $P$ touching one of the vertices but not the other, choose diagonals meeting at that vertex. For each $B$-face of $P$ touching neither $v_1$ nor $v_2$, choose diagonals randomly (among the two allowed arrangements). For a $B$-face of $P$ touching both $v_1$ and $v_2$, diagonals will be chosen in a later step.\label{enum:pickFourBfaces}
\setcounter{diagonalAssignmentSteps}{\value{enumi}}
\end{enumerate}
Let us examine the dependency cycles still left after this step. We cyclically label the $B$-faces in such a dependency cycle $f_0, \dots, f_{k-1}$ such that the $B$-face $f_0$ for which we have not chosen diagonals yet is adjacent to an $A$-face glued to a $B$-face $f_1$ for which we have not chosen diagonals yet, and so on until we are back at $f_0=f_k$. We call the edge of a $B$-face $f_j$ that is shared with the $A$-face of the same dodecahedron the ambiguous edge, see Figure~\ref{fig:twoDode}. The figure also shows that the ambiguous edge of $f_j$ and of $f_{j+1}$ are identified to become one edge in the cell complex which we also call an ambiguous edge. The faces $f_0,\dots,f_{k-1}$ form a fan about that ambiguous edge in the cell complex. When regarded as open interval, this edge in the cell complex has two distinct ends even if those two ends might correspond to the same ideal vertex.
\begin{enumerate}
\setcounter{enumi}{\value{diagonalAssignmentSteps}}
\item Pick an end $v$ for each ambiguous edge in the cell complex. Choose the diagonals for all the $B$-faces about the ambiguous edge such that the diagonals meet at $v$.
\setcounter{diagonalAssignmentSteps}{\value{enumi}}
\end{enumerate}
The recursive procedure already outlined above will now assign diagonals to all remaining $B$-faces:
\begin{enumerate}
\setcounter{enumi}{\value{diagonalAssignmentSteps}}
\item Whenever there is a dodecahedron $P$ such that we have already chosen diagonals on the (face glued to the) $A$-face of $P$ but not for some of the $B$-faces of $P$, pick diagonals for these $B$-faces as in Figure~\ref{fig:AalreadyChosen}.
\setcounter{diagonalAssignmentSteps}{\value{enumi}}
\end{enumerate}
\end{proof}

\bibliographystyle{hamsalpha}
\bibliography{dodecahedralBib}

\providecommand{\bysame}{\leavevmode\hbox to3em{\hrulefill}\thinspace}
\providecommand{\href}[2]{#2}
\providecommand{\eprint}{\begingroup \urlstyle{rm}\Url}
\begin{thebibliography}{LST08}

\bibitem[EP88]{EP}
D.~B.~A. Epstein and R.~C. Penner, \emph{Euclidean decompositions of noncompact
  hyperbolic manifolds}, J. Differential Geom. \textbf{27} (1988), no.~1,
  67--80.

\bibitem[Goe17]{goerner:platCensus}
Matthias Goerner, \emph{A census of hyperbolic {P}latonic manifolds and
  augmented knotted trivalent graphs}, New York Journal of Mathematics
  \textbf{23} (2017), 527--553, \eprint{arXiv:1602.02208}.

\bibitem[LST08]{LuoSchleimerTillmann:virtualGeometricTrig}
Feng Luo, Saul Schleimer, and Stephan Tillmann, \emph{Geodesic ideal
  triangulations exist virtually}, Proc. Amer. Math. Soc. \textbf{136} (2008),
  no.~7, 2625--2630, \eprint{arXiv:math/0701431}.

\bibitem[PW00]{PetronioWeeks:partiallyFlatTrig}
Carlo Petronio and Jeffrey~R. Weeks, \emph{Partially flat ideal triangulations
  of cusped hyperbolic 3-manifolds}, Osaka J. Math. \textbf{37} (2000), no.~2,
  453--466.

\end{thebibliography}

\bigskip

\end{document}